\documentclass[12pt]{amsart}
\usepackage{amscd,amssymb}
\usepackage[all]{xy}
\usepackage{color}
\usepackage[notref,notcite]{showkeys}

\begin{document}
\newtheorem{prop}{Proposition}[section]
\newtheorem{thm}[prop]{Theorem}
\newtheorem{lemma}[prop]{Lemma}
\newtheorem{cor}[prop]{Corollary}
\newtheorem{Question}[prop]{Question}

\newtheorem{Example}[prop]{Example}
\newtheorem{Examples}[prop]{Examples}
\newtheorem{Remark}[prop]{Remark} 

\newcommand{\extto}{\xrightarrow}
\newcommand{\cA}{{\mathcal A}}
\newcommand{\cB}{{\mathcal B}}
\newcommand{\cD}{{\mathcal D}}
\newcommand{\cY}{{\mathcal Y}}
\newcommand{\cU}{{\mathcal U}}
\newcommand{\cJ}{{\mathcal J}}
\newcommand{\cI}{{\mathcal I}}
\newcommand{\cF}{{\mathcal F}}
\newcommand{\cX}{{\mathcal X}}
\newcommand{\cZ}{{\mathcal Z}}
\newcommand{\cP}{{\mathcal P}}
\newcommand{\cQ}{{\mathcal Q}}
\newcommand{\cL}{{\mathcal L}}
\newcommand{\cM}{{\mathcal M}}
\newcommand{\cK}{{\mathcal K}}
\newcommand{\cH}{{\mathcal H}}
\newcommand{\cT}{{\mathcal T}}
\newcommand{\cG}{{\mathcal G}}
\newcommand{\cE}{{\mathcal E}}
\newcommand{\cC}{{\mathcal C}}
\newcommand{\cO}{{\mathcal O}}
\newcommand{\dpp}{\prime\prime}
\newcommand{\br}{{\bf r}}
\newcommand{\bx}{$_{\fbox{}}$\vskip .2in}
\newcommand{\Gr}{\operatorname{\mathbf{Gr}}\nolimits}
\newcommand{\Span}{\operatorname{Span}\nolimits}
\newcommand{\gr}{\operatorname{\mathbf{gr}}\nolimits}
\newcommand{\Mod}{\operatorname{\mathbf{Mod}}\nolimits}
\newcommand{\smod}{\operatorname{\mathbf{mod}}\nolimits}
\newcommand{\add}{\operatorname{\mathbf{add}}\nolimits}
\newcommand{\Add}{\operatorname{\mathbf{Add}}\nolimits}
\newcommand{\End}{\mbox{End}}
\newcommand{\Hom}{\mbox{Hom}}
\renewcommand{\Im}{\mbox{Im}}
\newcommand{\tr}[2]{\operatorname{Tr_{#1}(#2)}\nolimits}

\newcommand{\pd}{\operatorname{pd}\nolimits}
\newcommand{\id}{\operatorname{id}\nolimits}
\newcommand{\soc}{\operatorname{Soc}\nolimits}
\newcommand{\rad}{\operatorname{Rad}\nolimits}
\newcommand{\Top}{\operatorname{Top}\nolimits}
\newcommand{\Ker}{\mbox{Ker}}
\newcommand{\Coker}{\mbox{Coker}}
\newcommand{\coker}{\mbox{coker}}
\newcommand{\Tor}{\mbox{Tor}}
\renewcommand{\dim}{\mbox{dim}}
\newcommand{\gldim}{\operatorname{gl.dim}\nolimits}
\newcommand{\findim}{\operatorname{fin.dim}\nolimits}
\newcommand{\Ext}{\operatorname{Ext}\nolimits}
\newcommand{\op}{^{op}}
\newcommand{\HH}{\operatorname{HH}\nolimits}
\newcommand{\pr}{^{\prime}}
\newcommand{\f}{\operatorname{fin}}
\newcommand{\Sy}{\operatorname{syz}}
\newcommand{\semi}{\mathbin{\vcenter{\hbox{$\scriptscriptstyle|$}}
\;\!\!\!\times }}
\newcommand{\fralg}{K\!\!<\!\!x_1,\dots,x_n\!\!>}
\newcommand{\Efin}{\operatorname{Efin}\nolimits}
\newcommand{\Lin}{\operatorname{Lin}\nolimits}
\newcommand{\Tlin}{\operatorname{Tlin}\nolimits}
\newcommand{\WKS}{\operatorname{WKS}\nolimits}
\newcommand{\cx}{\operatorname{cx}\nolimits}
\newcommand{\pcx}{\operatorname{pcx}\nolimits}
\newcommand{\tpcx}{\operatorname{tpcx}\nolimits}
\newcommand{\mpcx}{\operatorname{mpcx}\nolimits}
\newcommand{\tmpcx}{\operatorname{tmpcx}\nolimits}
\newcommand{\mto}{\hookrightarrow}
\newcommand{\AR} {Auslander-Reiten }
\newcommand{\gk}{\operatorname{GKdim}\nolimits}
\newcommand{\tip}{\operatorname{Tip}\nolimits}
\newcommand{\sgn}{\operatorname{sgn}\nolimits}
\newcommand{\grb}{Gr\"obner }

\newcommand{\rank}{\operatorname{rank}\nolimits}

\newcommand{\rep}[1]{\operatorname{\mathbf{Rep}}\nolimits(K,{#1})}
\newcommand{\srep}[1]{\operatorname{\mathbf{rep}}\nolimits(K,{#1})}
\newcommand{\Ag}{\cA_{\Gamma}}
\newcommand{\Agw}{\cA_{\Gamma_W}}
\newcommand{\CC}[4]{C^{{#1},{#2}}_{{#3},{#4}}}
\newcommand{\ecl}{e_{\cL}}
\newcommand{\eclp}{e'_{\cL}}

\title[Convex subquivers]
{Convex subquivers and the finitistic dimension}

\author[Green]{Edward L.\ Green}\address{Department of
Mathematics\\ Virginia Tech\\ Blacksburg, VA 24061\\
USA} \email{green@math.vt.edu}

\author[Marcos]{Eduardo N.\ Marcos}\address{Departamento de Matem\'atica\\
IME-USP, Universidade de S\~ao Paulo\\ Rua do Mat\~ao 1010\\
S\~ao Paulo SP, Brazil,  05586060}
\email{enmarcos@gmail.com}

\thanks{\footnotesize The second named author  has been supported by the thematic project of Fapesp
2014/09310-5. }

\begin{abstract} Let $\cQ$ be a quiver and $K$ a field.
We study the interrelationship of homological properties of algebras associated to convex subquivers of  $\cQ$ and  quotients of the path algebra $K\cQ$.  We introduce the
homological heart of $\cQ$ which is a particularly nice convex subquiver of $\cQ$.
For any algebra of the form $K\cQ/I$,  the algebra associated to $K\cQ/I$ and
the homological heart have similar homological properties.  We give an application
showing that
the finitistic dimension conjecture need only be proved for algebras with path connected
quivers.  \end{abstract}
%\today
\maketitle

\section{Introduction}The study of homological properties of algebras is made both interesting
and challenging by the fact that such properties are usually not inherited by subalgebras or
quotient algebras.  In the rare cases where homological  properties are inherited, one can
hopefully pass to a quotient or subalgebra for which the question being studied
is easier to attack.   The finitistic dimension conjecture for finite dimensional algebras \cite{B} is
one such problem.  The \emph{finitistic dimension} of a finite dimensional algebra is the supremum
of the projective dimensions of finitely generated 
modules having finite projective dimension and the Finitistic
Dimension Conjecture asserts that this supremum is finite.  The conjecture has been proven directly in 
a number of cases, for example, \cite{GKK, GZ,IT, AR},  and attempts have been made to study the conjecture
by reducing the conjecture to `simpler' rings \cite{HLM,C1,C2,C3,Ch}, but the conjecture
remains open.   A survey on the state of the art until 2007 can be found in \cite{C3}.

In this paper, we study rings of the form $K\cQ/I$ where $K$ is a field, $\cQ$ a quiver
and $I$ is an ideal.   We find conditions on $\cQ$, independent of $I$, so that there is a full subquiver
$\cL$ such that the homological properties of the algebra $K\cL/(I\cap K\cL)$  provide 
information about homological properties of $K\cQ/I$.   In particular, if $e'$ is the sum of the vertices not in $\cL$, we investigate
conditions that imply that the canonical a surjection $K\cQ/I\to (K\cQ/I)/\langle e'\rangle$ is
a homological epimorphism; that is, the  surjection induces a
fully faithful functor $\cD^b(\text{mod}(K\cQ/I)\to\cD^b(\text{mod}((K\cQ/I)/\langle e'\rangle))$ \cite{GL, K, APT,P,GP}.

Let $\cQ$ be a finite quiver.   We let  $\cL$ be a full convex
subcategory of $\cQ$; that is, if $p$ is a path from  a vertex in $\cL$ to a vertex in $\cL$,
then each arrow
and vertex of $p$ is in $\cL$.      The study of
convex subcategories is not new and they have appeared in the study of
the representation theory of finite dimensional algebras; for example, see
\cite{SK, AL, ACMT, DR}.  Noting that if $\cL$ is convex, then
$(K\cQ/I)/\langle e'\rangle$ is isomorphic to $K\cL/(I\cap K\cL)$,
(see Proposition \ref{prop-2def}).
we show that if $\cL$ is a convex subquiver of $\cQ$, then the canonical
surjection $K\cQ/I\to K\cL/(I\cap K\cL) $ is a homological embedding,
see Theorem \ref{thm-convex-ext}.

In Section \ref{homo heart}, for any quiver
 $Q$, we define a subquiver which we will call the homological heart of $\cQ$.
We show that there is a strong homological relation between a finite dimensional algebra 
$K\cQ/I$ and the algebra $K\cH/(I\cap K\cH)$, where $\cH$ is the homological heart
of $\cQ$. For instance the
finitistic dimension of one is finite if and only if the finitistic dimension of the other is
finite.   A similar concept was defined and studied in \cite{LMM}.

The paper ends with an application showing that the  Finitistic Dimension Conjecture only needs
to be proved in the case where the quiver is path connected.

If $\cQ$ is a quiver then $\cQ_0$ denotes the vertex set of $\cQ$,  $\cQ_1$
denotes the arrow set of $\cQ$.  If $K$ is a field, $K\cQ$ denotes the path algebra.
If $J$ denotes the ideal in a path algebra $K\cQ$ generated by the arrows of
$\cQ$, then an ideal $I$ in $K\cQ$ is \emph{admissible} if $J^n\subseteq
I\subseteq J^2$, for $n\ge 2$.  If $\Lambda$ is an algebra and $A$ is a subset of $\Lambda$,
then $\langle
A\rangle$ denotes the ideal in $\Lambda$ generated by $A$.  Unless otherwise
stated, modules are right modules and if $a$ and $b$ are arrows in a quiver $\cQ$,
$ab$ denotes the concatenation of first $a$ then $b$.

\section{Definitions and basic properties}
Throughout this      paper
let $K$ be a field, $\cQ$ a finite quiver, and let $K\cQ$ denote the path
algebra.   We say a subquiver $\cL$ of $\cQ$ is \emph{full} if for any two
vertices $v$ and $w$ of $\cL$,  all the arrows in $\cQ$ with
origin $v$ and terminus $w$ are also arrows in $\cL$.  A full subquiver is uniquely determined
by its vertex set and all subquivers considered in this paper are assumed to be full. 
If $X\subset \cQ_0$, then  the full subquiver of
$\cQ$ having vertex $X$  is called the
\emph{subquiver generated by $X$}. We freely identify a set of vertices with the full subquiver
it generates.  The \emph{empty subquiver} is defined
to be the subquiver with no vertices or arrows.
If $\cL$ and $ \cM$ are subquivers of $\cQ$, then $\cL\cup \cM$  (respectively,
$\cL\cap\cM$) is the full subquiver
generated by $\cL_0\cup \cM_0$ (resp., by $\cL_0\cap\cM_ 0$).

A fundamental concept that plays a central
role in this paper  is convexity.
 We say a full subquiver $\cL$ of $\cQ$ is \emph{convex}
if for any two vertices $v, w$ in $\cL$ and for any path $p$ from $v$ to $w$, every vertex occuring in $p$
is in $\cL$.
We introduce 3 other subquivers associated to a given subquiver.
Let $\cL$ be a full subquiver of $\cQ$.   Define $\cL^+$ and $\cL^-$ to be the full subquivers
such that the vertex set of $\cL^+$ is
\[\{v\in\cQ_0\mid v\notin\cL_0 \text{ and }
 \text{ there is a path in }\cQ\text{ from a vertex in }\cL_0\text{ to }v\}\] and the vertex set
of $\cL^-$ is
\[\{v\in\cQ_0\mid v\notin\cL_0 \text{ and }
\text{ there is a path in }\cQ\text{ from }v\text{ to a vertex in }\cL_0\}.\]

A third construction is  $\cL^o$ whose vertex set is
\[
\{v\in\cQ_0\mid v\notin\cL_0 \text{ and there are no paths between }v\text{ and any
vertex in }\cL\}.\]

It is easy to see that if $\cL$ is a full subquiver of $\cQ$  then
$\cQ=\cL\cup\cL^+\cup \cL^-\cup\cL^o$ and
$\cL$ is convex if and only if $\cL^+_0\cap\cL^-_0$ is empty.
In particular, we have 
\begin{prop}\label{prop-empty} Let $\cL$ is a subquiver of $\cQ$.  If either
$\cL^+$ or $\cL^-$ is empty then $\cL$ is convex.
\end{prop}

\section{Convexity and algebras} 

For the remainder of this paper, $\cL$ will denote a full subquiver of $\cQ$ and $I$ is an ideal
in $K\cQ$ contained in the ideal  generated
by the paths of length 2. We set $\Lambda = K\cQ/I$.  
Let $e_{\cL}$ be the
idempotent associated to the 
sum of the vertices in $\cL$ and $e'_{\cL}$ be the idempotent associated to the
sum of vertices not in $\cL$.  Then
$e_{\cL}$ and $\eclp$ are orthogonal idempotents in $K\cQ$ and  $1=\ecl+\eclp$.
We abuse notation and view $\ecl$ and $\eclp$ as idempotents in both $K\cQ$
and $\Lambda$.   The \emph{algebra associated to the pair $(\cL,\Lambda)$} is defined to
be $\Gamma=\Lambda/\langle \eclp\rangle$, where $\langle \eclp\rangle$ is
the ideal generated by $\eclp$.
  Another possible choice of the algebra associated to $(\cL,\Lambda)$ could have been
$\ecl\Lambda\ecl$.  In general, $ \Lambda/\langle \eclp\rangle$ and 
$\ecl\Lambda\ecl$ are not isomorphic.  For example, if $\cQ$ is the quiver
$v\to w\to x$, $\Lambda=K\cQ$, and $\cL$ is the subquiver generated by $\{v,x\}$, then
$\ecl\Lambda\ecl$ is the hereditary algebra with quiver $v\to x$ of dimension 3 while
$\Lambda/\langle \eclp\rangle$ is the semisimple algebra of dimension 2 corresponding to
the two vertices $v$ and $x$ with no arrows.  Before proving that when
$\cL$ is convex, the two algebras are in fact isomorphic, we present the following  useful
result.  We simplify notation by letting $e=\ecl$ and $e'=\eclp$ when no confusion can
arise.

\begin{lemma}\label{lem-prep}Suppose that $\cL$ is a convex subquiver
of $\cQ$.   Then if $\lambda,\gamma\in\Lambda$,
 $e\lambda e\gamma e
=e \lambda\gamma e$.  In particular,  $e\lambda e'\gamma e=0$.
\end{lemma}

\begin{proof} Let $\lambda$ and $\gamma$ be elements of $\Lambda$.  If $\pi\colon K\cQ\to \Lambda$
is the canonical surjection, then there are paths $p_i$ and $q_j$ and constants $c_i$ and $d_i$, so
that $\lambda=\pi(\sum_ic_ip_i)$ and $\gamma=\pi(\sum_jd_jq_j)$.
Then, using that $1=e+e'$, we have
\[
e\lambda\gamma e= e\lambda(e+e')\gamma e=\]\[
 \pi(e(\sum_i c_ip_i)e(\sum_jd_jq_j)e)+ \pi(e(\sum_i c_ip_i)e'(\sum_jd_jq_j)e).\]
Since $\cL$ is convex,  each $e p_ie' q_je =0$ and the result follows.

\end{proof}

\begin{cor}\label{cor-ringh} Keeping the hypothesis of Lemma \ref{lem-prep},
the homomorphism $\phi\colon \Lambda\to e \Lambda e$ defined by
$\phi(\lambda)=e\lambda e$ is a ring surjection.
\end{cor}

The convexity of $\cL$ is essential in the last part of the next result.

\begin{prop}\label{prop-2def} Let $\cL$ be a full subquiver of $\cQ$ and
$\Lambda=K\cQ/I$ where $I\subseteq  J^2$.  The following three $K$-algebras are isomorphic.
\begin{enumerate}
\item $e\Lambda e$.
\item $e K\cQ e /e I e$
\item     $e K\cQ e /I\cap e K\cQ e $
\item
If $\cL$ is convex then 
 $\Lambda/\langle e'\rangle$,  the algebra associated to
 $(\cL,\Lambda)$, is isomorphic to $e\Lambda e$ This isomorphism sends $\bar \lambda$ to $e \lambda e$,  where
$\lambda\in\Lambda$ and $\bar \lambda$ denotes the image of $\lambda$ in
$\Lambda/\langle e' \rangle$.
\item If $\cL$ is convex, then $K\cL/(I\cap K\cL)$ is isomorphic
to $\Lambda/\langle e'\rangle$.
\end{enumerate}
\end{prop}

\begin{proof} The proofs
of parts (1)-(3) are left to the reader.  By Corollary \ref{cor-ringh}, the map
$\lambda$ to $e\lambda e$ is a ring surjection.  We claim $\langle e'\rangle$ is
the kernel of this homomorphism.   Since $ee'e=0$, $\langle e'\rangle$ is contained
in the kernel.  Now suppose that $\lambda$ is in the kernel; that is, $e\lambda e=0$.
Now $\lambda=e\lambda e +e'\lambda e +e\lambda e' +e'\lambda e'$.  Since
$e'\lambda e +e\lambda e' +e'\lambda e'\in \langle e'\rangle$, we are done.
\end{proof}

Consider the map $\mu\colon \ecl\Lambda\ecl\to\Lambda$ given by inclusion
and the map  $\nu\colon \Lambda \to  \ecl\Lambda\ecl$ given by $\nu(\lambda)=\ecl \lambda \ecl$. If $\cL$ is convex,  then $\nu\circ \mu = 1_{ \ecl\Lambda\ecl}$.  In fact we have the following result.

\begin{prop}\label{prop-split} Let $\cL$ be a convex subquiver of $\cQ$
 and let $\pi\colon\Lambda
\to \Gamma=\Lambda/\langle e'\rangle$ 
denote the canonical surjection.  Then there is a ring (without 1) splitting
of
$\phi$
mapping $ \Lambda/\langle e'\rangle$ to $\Lambda$ 
 given as follows: if $\lambda\in\Lambda$ then $\lambda +\langle e'\rangle$ is sent to
$e \lambda e$.
\end{prop}

\begin{proof} Let $\delta\colon
\Lambda/\langle e'\rangle\to e\Lambda e$ be the isomorphism described in
Proposition  \ref{prop-2def}.  Then 
$incl\circ\delta\colon \Lambda/\langle e'\rangle\to \Lambda$
is easily seen to  splitting  to the canonical surjection $\psi\colon\Lambda\to\Lambda/\langle e'\rangle$.   The reader may check that $incl\circ\delta$ is a ring homomorphism
taking $1\in  \Lambda/\langle e'\rangle$ to $e\in\Lambda$.
\end{proof}

\section{Module theoretic results}\label{sect-prelim}

Recall that a ring epimorphism $A\to B$ is called a \emph{homological
epimorphism} if the induced map $\cD^b(\text{mod}(B))\to
\cD^b(\text{mod}(A))$ is fully faithful \cite{GL,K}.  Equivalently,
if $A\to B$ is a homological epimorphism and $M,N\in \text{mod}(B)$,
then $\Ext_B^n(M,N)$ is isomorphic to $\Ext_A^n(M,N)$
for all $n\ge 0$. Consequently, if $A$ is a finite dimensional algebra and
$A\to B$ is a homological epimorphism, then $\gldim(A)\ge \gldim(B)$ and
$\findim(A)\ge\findim(B)	$.																			

In this section we show that the canonical surjection $\Lambda\to\Gamma $ is 
a homological epimorphism if $\cL$ is convex subquiver of $\cQ$.  We begin
by assuming $\cL^+$ is empty.
As examples,  if $\cL$ is a full subquiver of
$\cQ$, then both $(\cL\cup \cL^+)^+ $ and $(\cL\cup \cL^-)^-$ are empty 
and hence $\cL\cup\cL^+$ and $\cL\cup\cL^-$ are convex.
Another example would be $\cL^+$ which is always convex and $(\cL^+)^+$ is always empty.

We will apply the following result.

\begin{prop}\label{prop-epi} (1) If $\cL^+$ is empty, then
the canonical surjection $\Lambda\to \Gamma$ is a homological
epimorphism.

(2).  If $\cQ=\cL\cup\cL^+$,  then $\Gamma$ is a projective
left $\Lambda$-module and
the canonical surjection $\Lambda\to \Gamma$ is a homological
epimorphism
\end{prop}

\begin{proof}
First assume that $\cL^+$ is empty.  Then $e\Lambda e'=0$
and hence
 $\Lambda =\left(\begin{array}{cc}\Gamma & 0\\e'\Lambda e&e'\Lambda e'\end{array}
\right)$. Part (1) follows from \cite{FGR, PR}.

Next , assume that $\cQ=\cL\cup\cL^+$.  We claim that
$\Gamma$ is a left projective $\Lambda$-module. 
Note  that $e'\Lambda e=(0)$ since there are no paths from vertices in $\cL^+$ to
$\cL$ by  convexity of $\cL$.   By Proposition \ref{prop-split}
there is splitting $\psi\colon \Gamma\to\Lambda$ of the canonical
surjection $\Lambda\to\Gamma$ given by $\psi(\lambda+\langle e'\rangle)=e\lambda e$.  Using that $e'\Lambda e=0$,
one sees that 
\[\lambda\psi(\lambda')=\lambda(e\lambda' e)=e\lambda e\lambda'e=
e\lambda\lambda' e=\psi(\lambda\lambda'+\langle e'\rangle).\]
Hence $\psi$ is a left $\Lambda$-module homomorphism. 
We conclude that $\Gamma$ is a left projective $\Lambda$-module.
Part (2) follows from \cite{APT}.

\end{proof}

Using the previous proposition, we prove the main result.

\begin{thm}\label{thm-convex-ext} Let $K$ be a field, $\cQ$ a finite quiver, and $\Lambda=K\cQ/I$, where
$I$ is an ideal in $K\cQ$ contained in the ideal generated by the paths of length 2 in $K\cQ$.  
Suppose that   $\cL$ is a convex subquiver
of $\cQ$ and let $\Gamma$ be the algebra associated to $(\cL,\Lambda)$.
Then the canonical surjection $\Lambda\to\Gamma$ is a
homological epimorphism.

\end{thm}

\begin{proof}

We begin by applying Proposition \ref{prop-epi} as follows.  Let $\cM=\cL\cup\cL^+$.
We see that $\cM^+$ is empty.  Let $\Delta$ be the algebra associated to $(\cM,\Lambda)$.  Proposition \ref{prop-epi}(1) yields
the canonical surjection $\Lambda\to \Delta$ is a homological
epimorphism.

Next we assume that $\cL$ is convex. View $\cL$ as a subquiver of $\cM$ and let $\Sigma$ be the algebra associated
to $(\Delta,\cL)$. 
We apply Proposition \ref{prop-epi}(2) the canonical surjection
$\Delta\to\Sigma$ is a homological epimorphsim.

 It is straightforward to show
$\Sigma$ is isomorphic to $\Gamma$ and the composition of
homological epimorphisms is a homological
epimorphism. 
\end{proof}

\begin{cor}\label{cor-epi} Suppose that $\Lambda$ is a finite dimensional $K$-algebra.
If $\cL$ is a convex subquiver of
$\cQ$, then
\begin{enumerate}\item $\gldim(\Lambda)\ge\gldim (\Gamma)$
\item $\findim(\Lambda)\ge\findim(\Gamma)$.
\end{enumerate}
\end{cor}

\section{The homological heart}\label{homo heart}

In this section we introduce the homological heart of
a quiver.
Let 
\[X=\{v\in\cQ_0 \mid v\text{ is a vertex in a nontrivial cycle in }\cQ\}\] and  let 
\[Y=X\cup \{y\in\cQ_0\mid y \text{ is a vertex in a path with origin}\]
\[ \text{ and terminus vertices in }X\}.\]
Let $\cH(\cQ)$, or simply  $\cH$ when no confusion can arise, be the subquiver
of $\cQ$ with vertex set $Y$.  We call $\cH$ the \emph{homological heart of $\cQ$}.  

The proof of the next result is left to the reader.

\begin{prop}\label{prop-hh1}  Let $\cH$ be the homological heart of $\cQ$.
Then the following statements hold.
\begin{enumerate}
\item The subquiver $\cH$ is the empty quiver if and only if $\cQ$ contains no nontrivial
cycles; that is, $\cQ$ is triangular.

\item The subquiver $\cH$ is the smallest convex subquiver of $\cQ$ that contains all
the nontrivial path connected components of  $\cQ$.
\item The homological heart of $\cQ$ is an invariant of $\cQ$.
\item The subquiver $\cH^+\cup\cH^-\cup\cH^o$ contains no oriented cycles.
\end{enumerate}
\end{prop}

Throughout this section,
we let $\Gamma$ denote the algebra associated to 
 $(\cH,\Lambda)$. The goal of this section is to show that  $\Gamma$ provides information about
a number homological features of $\Lambda$ .

We fix the following notation. 
For each vertex $v\in\cQ$, let $P_v$ denote the projective $\Lambda$-module $v\Lambda$ and
$I_v $ the injective $\Lambda$-envelope of the simple one dimensional $\Lambda$-module associated
to the vertex $v$.
 
 We begin with a known result 
which we include for completeness.  

\begin{prop}\label{prop-res-path}
(1). Let $P^n\to P^{n-1}\to\cdots\to P^2\to P^1\to P^0$
be a sequence of nonzero $\Lambda$-homorphisms with each $P^i$ a projective
$\Lambda$-module which is a direct sum of projective modules of the form $P_v$.  Assume that
 the image of $P^i\to  P^{i-1}$ is contained 
$P^{i-1}\br$, where $\br$ is the Jacobson radical of $\Lambda$ and that $\ker(P^n\to P^{n-1})$ is contained in $P^n\br$, for
all $1\le i\le n$.  
If $v\in\cQ_0$ is
such that  $v\Lambda $ is a summand of $P^n$, then there exist a vertex $w$ and a 
path of length $\ge n$ from $w$ to $v$ in $\cQ$ such that  $w\Lambda$ is summand of  $ P^0$.

(2) Let $I_0\to I_{1}\to\cdots\to I_{n-1}\to I_n$
be a sequence of nonzero $\Lambda$-homorphisms with each $I_i$ an injective
$\Lambda$-module which is a direct sum of injective modules of the form $I_v$.  Assume that
 $\soc(I_i)$ is contained in the image of $I_{i-1}\to  I_{i}$  and
 that $\soc(I_{i-1})$ is contained in the kernel of $I_{i-1}\to I_i$, for $1\le i\le n$.
If $v\in\cQ_0$ is
such that  $I_v $ is a summand of $I_n$, then there exist a vertex $w$ and a 
path of length $\ge n$ from $v$ in $\cQ$ to $w$ such that  $I_w\Lambda$ is summand of  $ I_0$
\end{prop}

\begin{proof}(Sketch) We only sketch a proof of (1) leaving (2) to
the reader.
 The proof follows by induction and noting that if $v$ is in the support of the $\Lambda$-module $ P_w\br$, then there is a  path from $w$ to $v$ of length at least $1$ and
if $f\colon P_v \to P_w\br$ is a nonzero map, then $v$ is in the support of $P_w\br$. 
\end{proof}
By Proposition \ref{prop-hh1}(4), we set $t$ to be the longest path in $\cQ$ with support
in $\cH^-\cup\cH^0\cup\cH^+$.  The next result will be used often.

\begin{lemma}\label{lem-dist} Let $M$ be a $\Lambda$-module.
Then, for $\ell> t$ \begin{enumerate}
\item the $\ell$-th syzygy of a $\Lambda$-module has support in
$\cH\cup\cH^+$ and 
\item the $\ell$-th cosyzygy of $\Lambda$-module has
support in $\cH\cup\cH^-$.

\end{enumerate}
\end{lemma}

\begin{proof}
We only prove (1).  Let $M$ be a $\Lambda$-module and $
 \cdots\to  P^2\to P^1\to P^0\to M\to 0$ be a  minimal projective  
$\Lambda$-resolution of  $M$.    From the definitions of $\cH^-$, $\cH^+$ and $\cH^0$, 
any path ending at a vertex in $\cH^-\cup\cH^0$ has support in
$\cH^-\cup\cH^0$.  Suppose $\ell >t$.  Suppose that there is a nonzero
map from $P_v$ to $P^{\ell}$ and that $v$ is
vertex in $\cH^-\cup\cH^0$.
Then there would be a path of length $\ge \ell$ ending at $v$  by Proposition \ref{prop-res-path}.
This path has support in $\cH^-\cup\cH^0$. By our assumption on $t$, we
must have  $\ell\le t$, a contradiction.  Hence if $P_v$ is a summand of $P^{\ell}$,
$v$ must be a vertex in $\cH\cup\cH^+$.   It follows that $P^{\ell}$ has support
in $\cH\cup\cH^+$ and we are done.
\end{proof}

\begin{cor}\label{cor-ext} Let $M$ and $N$ be $\Lambda$-modules.  Then there are
$\Lambda$-modules $A$ and $B$ such that, $\ell\ge 2t+3$,
\[\Ext_{\Lambda}^{\ell}(M,N)\cong \Ext_{\Lambda}^{\ell-2t-2}(A, B)\] with
$A$ having support in $\cH\cup \cH^+$ and $B$ having support in $\cH\cup\cH^-$.

\end{cor}

\begin{proof}
Take $A$ to be the $t+1$-th syzygy of $M$ and $B$ to be the $t+1$-th cosyzygy of $N$ and
apply Lemma \ref{lem-dist} and dimension shift.

\end{proof}

The next result considers
modules with support contained in either $\cH\cup\cH^+$ or $\cH\cup\cH^-$.

\begin{prop}\label{prop-restrict}  Suppose that $A$ is a $\Lambda$-module with support
contained in $\cH\cup\cH^+$ and $B$ is a $\Lambda$-module with support
contained in $\cH\cup\cH^-$.  If 
$\cdots\to  P^2\to P^1\to P^0\to A\to 0$  is a minimal projective  
$\Lambda$-resolution of $A$ and
$ 0\to B\to I_0\to I_1\to I_2\to\cdots$ is a minimal injective  
$\Lambda$-coresolution of $B$,   then each $P^n$ has support contained
in $\cH\cup\cH^+$ and each $I_n$ has support contained in $\cH\cup\cH^-$.

\end{prop}

\begin{proof}  We use the observation that any path ending at a vertex in $\cH^-\cup\cH^0$
(respectively, beginning at a vertex in $\cH^+\cup\cH^0$)
has support in $\cH^-\cup\cH^0$ (resp., $\cH^+\cup\cH^0)$.  The result follows from
Proposition  \ref{prop-res-path}.

\end{proof}

Let $C$ be a $\Lambda$-module.  Define $C^+$ to be the largest submodule of $ C $
 having support contained in $\cH^+$ and $C_-$ be the smallest submodule of $C$ such
that $C/C_-$ has support contained $\cH^-$.
Let $e$ be the sum of the vertices in $\cH$, $e^+$ be the sum of the vertices in $\cH^+$,
$e^-$ be the sum of the vertices in $\cH^-$, and $e^0$ be the sum of the vertices in $\cH^0$.
Letting $\Gamma$ be the algebra associated to $(\cH,\Lambda)$, we have the
following result, whose proof is left to the reader.

\begin{lemma}\label{lem-plus} Let $C$ be a $\Lambda$-module.
\begin{enumerate}
\item  $C^+$ is the trace of $e^+\Lambda$ in $C$.
\item If $C$ has support in $\cH\cup\cH^-$, then $C_-$ is the trace of $e\Lambda$ in $C$.
\item If $C$ has support in $\cH\cup\cH^+$, then $C/C^+$ has support contained in $\cH$ and
hence is a $\Gamma$-module.
\item If $C$ has support in $\cH\cup\cH^-$, then $C_-$ has support in $\cH$ and hence is
a $\Gamma$-module.
\item If $C$ has support in $\cH\cup\cH^+$, then  $C^+=Ce^+\Lambda$
\item  If $C$ has support in $\cH\cup\cH^+$, then $C/C^+
\cong C\otimes_{\Lambda}\Gamma$.

\end{enumerate}\end{lemma}

%{\tt resolution result}
Before proving a
result that relates projective $\Lambda$-resolutions of $\Lambda$-modules $A$
whose support is contained in $\cH\cup\cH^+$ and projective $\Gamma$-resolutions
of $A/A^+$, from  Proposition \ref{prop-epi}(2), we know
that      
$\Lambda e^+$ is a left projective $\Lambda$-module
and hence $\Gamma$ is a left projective $\Lambda$-module. Noting
that $\Lambda e^+\Lambda=\Lambda e^+$, we see that 
$\Lambda e^+\Lambda$ is a \emph{strong idempotent ideal} as defined
in \cite{APT} and that
the following result is similar to Theorem 1.6 of \cite{APT}. 

\begin{prop}\label{prop-+resol} Let $A$ be a $\Lambda$-module whose support
is contained in $\cH\cup\cH^+$.  
If $\cdots\to P^2\to P^1\to P^0\to A\to 0$ is a minimal projective
$\Lambda$-resolution of $A$, then
\[\cdots\to P^2/( P^2)^+\to P^1/( P^1)^+\to P^0/( P^0)^+\to A/A^+\to 0\]
is a minimal projective $\Gamma$-resolution of $A/A^+$.
\end{prop}

\begin{proof}
By  Proposition \ref{prop-epi}(2), we see that
$\Gamma$ is a left projective $\Lambda$-module hence 
tensoring with $\Gamma$ is exact. Since $A/A^+\cong A\otimes_{\Lambda}\Gamma$
and $P^n/(P^n)^+\cong P^n\otimes_{\Lambda}\Gamma$ and minimality
is clear, the result follows. 

\end{proof}

We remark that there is a ``dual'' result for injective $\Lambda$-coresolutions of $\Lambda$-modules
$B$ whose support is contained in $\cH\cup \cH^-$ and injective $\Gamma$-coresolutions
of $B_-$.

\begin{prop}\label{prop-hom} Let $A$ and $B$ be $\Lambda$-modules with support
contained in $\cH\cup\cH^+$ and $\cH\cup\cH^-$ respectively.  Then, for $n\ge 0$,
\[ \Ext_{\Lambda}^n(A, B)\cong\Ext_{\Gamma}^n(A/A^+,B_-)\]
\end{prop}

\begin{proof}  First we note that 
$\Hom_{\Lambda}(A, B)$ and $\Hom_{\Gamma}(A/A^+,B_-)$ are naturally
isomorphic.  To see this, if $f\colon A\to B$ is a $\Lambda$-homomorphism,
then, using Lemma \ref{lem-plus}(6), $f\otimes 1_{\Gamma}\colon A/A^+\to B/B^+$.  But $B^+=(0)$ since
$B$  has support in $\cH\cup\cH^-$ and $\cH$ is convex. 
Hence, $f\otimes 1_{\Gamma}\colon A/A^+\to B$.
Next we note that composing $f\otimes 1_{\Gamma}$ with 
the canonical surjection  $B\to B/B_-$ is $0$ since
$B/B_-$ has support contained in $\cH^-$ and $A/A^+$ has support contained
in $\cH$.   We conclude that the image of $f\otimes 1_{\Gamma}$ is contained
in $B_-$.  Thus, we have constructed a map from $\Hom_{\Lambda}(A,B)$ to
$\Hom_{\Gamma}(A/A^+,B_-)$.  To see that this map  is a monomorphism
we note that $\Hom(A^+,B)=(0)$ since $A^+$ has support contained in $\cH^+$, and the convexity of $\cH$ implies that $\cH^+\cap\cH^-$ is
empty.

On the other hand, if $g\colon A/A^+\to B_-$ is a $\Gamma$-homomorphism,
then composing with the canonical surjection $A\to A/A^+$ and the inclusion
$B_-\to B$ we obtain a map from 
$\Hom_{\Gamma}(A/A^+,B_-)$ to $\Hom_{\Lambda}(A,B)$.  The maps defined
are inverse to each other and we are done.   

Using the naturality of the isomorphism \[\Hom_{\Lambda}(A,B)\to\Hom_{\Gamma}(A/A^+,B_-)\]
 and Proposition \ref{prop-+resol},
the result follows.

\end{proof}

The following result is the main result of this section.   For a $\Lambda$-module $L$,
 we let $ \Omega^n(L)$ denote the $n$-th syzygy of $L$ in a  minimal
projective $\Lambda$-resolution of $L$ and 
$ \Omega^{-n}(L)$ denote the $n$-th cosyzygy of $L$ in a  minimal
injective $\Lambda$-coresolution of $L$
 
\begin{thm}\label{thm-hh} Let $K$ be a field, $\cQ$ a finite quiver, and $\Lambda=K\cQ/I$, where $I$ is an
admissible ideal $I$ in $K\cQ$.  Let $\cH$ be the homological heart of $\cQ$ and
$\Gamma$ be the algebra associated to $(\Lambda,\cH)$.  Let $t$ be
length of the longest path in $ \cQ$ having support contained in $\cH^+\cup\cH^-\cup\cH^0$.
Then \begin{enumerate}
\item $\gldim(\Lambda)$ is finite if and only if $\gldim (\Gamma)$ is finite.
\item $\findim(\Lambda)$ is finite if and only if $\findim(\Gamma)$ is finite.
\item If $M$ and $N$ are $\Lambda$-modules and $\ell\ge 2t+3$, then
\[\Ext^{\ell}_{\Lambda}(M,N)\text{ is naturally  isomorphic
to }\Ext^{\ell-2t-2}_{\Gamma}(A_M, B_N),\] where
$A_M=\Omega^{t+1}(M)/(\Omega^{t+1}(M))^+$ and $B_N=\Omega^{-t-1}(N)_-$.

\end{enumerate}

\end{thm}

\begin{proof} Since $\cH$ is convex, applying
Theorem \ref{thm-convex-ext}, we see \sloppy that
 $\gldim(\Lambda)\ge \gldim (\Gamma)$ and
$\findim(\Lambda)\ge\findim(\Gamma)$.
Thus, if  $\gldim(\Lambda)$ is finite then $\gldim (\Gamma)$ is finite
and if $\findim(\Lambda)$ is finite then $\findim(\Gamma)$ is finite.
We note that part (3) implies that  $\gldim(\Lambda)\le \gldim (\Gamma)$ and
hence part (1) would follow.

Thus, we must show part (3) holds and show that $\findim(\Lambda)$ is finite if $\findim(\Gamma)$ is finite.
First, assume that $\findim(\Gamma)$ is finite. We prove that
$\findim(\Lambda)\le \findim(\Gamma)+t$.  Let $M$ be a $\Lambda$-module having
finite projective dimension.   Then $\Omega^t(M)$ has finite projective dimension over
$\Lambda$.  By Proposition \ref{prop-+resol}, the projective dimension of $\Omega^t(M)$ over
$\Lambda$
is the same as the projective dimension of $\Omega^t(M)/(\Omega^t(M))^+$ over $\Gamma$.
Thus the projective dimension of $\Omega^t(M)$ over $\Lambda$ $\le \findim(\Gamma)+t$ and (2) follows.

Part (3) follows from the proof of Corollary \ref{cor-ext} and Proposition \ref{prop-hom}.
\end{proof}

%************************
\section{An Application}\label{Application}

 We  recall the definition of path connectedness. A 
subquiver $\cL$ of $\cQ$ having vertex set $X$, is  \emph{path connected} if
for each pair of vertices $v$ and $w$ in $X$, there is  both a  directed
path from $v$ to $w$ in $\cL$ and a directed path from $w$ to $v$ in $\cL$.

Viewing a vertex as an oriented cycle at $v$ with no arrows,  called the 
\emph{trivial cycle at $v$}, we have an equivalence relation
on the vertices of $\cQ$ given by setting two vertices $v$ and $w$  equivalent
if there is an oriented cycle in $\cQ$ in which both vertices occur.  The equivalence classes of 
this relation are called the \emph{path connected components of $\cQ$}.   
We call a path connected component of $\cQ$ \emph{trivial} if it consists of one vertex $v$ 
and the only cycle in $\cQ$ on which
$v$ lies is the trivial cycle at $v$.
 A path connected component is called \emph{nontrivial}
if it contains at least one arrow.

  In this section we show that the finitistic dimension conjecture reduces to the case where the quiver is path connected.  More precisely, we have the following result.
  \begin{thm}\label{thm-appl}
If  the finitistic dimension 
of every algebra $K\cQ/I$, with $I$ admissible and $\cQ$ path connected, is finite, then the finitistic dimension
of every algebra $K\cQ/I$	is finite, for all choices of $\cQ$ and $I$ admissible . 
\end{thm}

We begin with a result about the finitistic dimension of triangular matrix rings, \cite{FGR}Theorem 4.20, also see \cite{PR}.
The Fossum-Griffith-Reiten theorem is more general and also provides bounds which we do not need.

\begin{thm}\label{thm-fgr}\cite{FGR}Let $A$ and $B$ be artin algebras and $M$ a finitely generated
$A$-$B$-bimodule.  If the finitisitic dimensions of  $A$ and $B$ are finite, then the finitistic
dimension of the upper triangular matrix ring $\left( \begin{array}{cc} A&M\\0&B\end{array}\right)$
is finite.
\end{thm}  

\emph{Proof of Theorem \ref{thm-appl}}.  Suppose that  the finitistic dimension 
of an algebra $K\cQ/I$ is finite if $I$ is admissible and $\cQ$ is path connected.
Now let  $\cQ$  be an arbitrary  finite quiver, $I$ be an admissible ideal in
$k\cQ$, and $\Lambda=K\cQ/I$.
%Without loss of generality, we assume that $\cQ$ is connected; that is, if $v$ and $w $ are
%vertices in $\cQ$ then there is a undirected path in the underlying graph of $\cQ$  
%having endpoints $v$ and $w$.
We wish to show that the finitistic dimension of $\Lambda$ is finite.

We proceed by induction on the number of nontrivial path    connected components. 
 If
there are no nontrivial path connected components,  then $\cQ$ contains no cycles and 
$\Lambda$ has finite global dimension.  Hence.
we are done. 

Suppose there is only one nontrivial path connected component, $\cH$.  Then $\cH$, being
convex, is the homological heart of $\cQ$.  
By Theorem \ref{thm-hh}, we may assume that $\cQ$ equals the homological heart $\cH$.  Since
we are assuming that $\cH$ is path connected, 
the hypothesis implies that the
finitistic dimension of  $\Lambda$ is finite.

Next suppose that $\cQ$ has at least $k$ nontrivial path connected components with $k\ge 2$.
By induction, we assume that if a quiver $\cQ'$ has $k-1$ nontrivial path connected
components and $I'$ is an admissible ideal in $K\cQ'$, then the finitistic dimension of
$K\cQ'/I'$ is finite.  The homological heart of $\cQ$, $\cH$, has $k$
nontrivial path connected components.
By Theorem \ref{thm-hh} we may assume that  $\cQ=\cH$.  

Consider the set $A$ of arrows of $\cQ$ that do not lie on a cycle in $\cQ$. 
We form a new graph $ \Delta$ whose vertex set is the set of the  connected path components
of $\cQ$ (including the trivial components) and the arrow set is $A$ where if  $a\in A $ with $a\colon v\to w$ in $ \cQ$ then, as an
arrow in $\Delta$, $a$ is an arrow from the   path connected component of $\cQ$ containing $v$ to
the path connected component of $\cQ$ containing $w$.  Since arrows in $A$ cannot lie on an
oriented cycle in $\cQ$, $\Delta$ contains no cycles. Thus, some vertex in  $\Delta$ is a source.  
  It follows that there is the path connected component
$\cL$ in $\cQ$
associated a source vertex in $\Delta$ such that $\cL$ is a path connected subquiver of $\cQ $ and
$\cL^-$ is empty.  

We claim $\cL$ is nontrivial.  If $\cL$ is trivial then $\cL_0=\{v\}$ for some
$v\in\cQ_0$.  Since $v$ is not contained in a cycle,  $v\in Y$ as defined in Section \ref{homo heart}.  Hence there is a path from
$v$ to a nontrivial path connected component to $v$.  But then $\cL^-$ is not empty
which is a contradiction.

Let $\cL$ be a nontrivial path connected
component of $\cQ$ such that $\cL^-$ is empty.    Let $\cM$ be the subquiver of $\cQ$ with
vertex set $\cQ_0\setminus \cL_0$, let $e$ be the idempotent associated to the sum of the
vertices of $\cL$, and $e'=1-e$. Note that $ \cM$ has $k-1$ path connected components.  
Then $\Lambda =\left(\begin{array}{cc}e\Lambda e& e\Lambda e'\\
   e'\Lambda e&e'\Lambda e' \end{array}\right)$.   Since  $\cL^-$ is empty, we see that
$e'\Lambda e=0$.   The fact that $e'\Lambda e=0$ implies that the quiver of $e'\Lambda e'$ is $\cM$; that
is, there is no path from a vertex in $\cM$ to a vertex in $\cM$ such that all the `internal' vertices
of the parth are in $\cL$.   

Summarizing, we have  $\Lambda =\left(\begin{array}{cc}e\Lambda e& e\Lambda e'\\
   0&e'\Lambda e' \end{array}\right)$ such that the quiver of $e\Lambda e$ is $ \cL$ and the quiver
of $e'\Lambda e'$ is $\cM$.  We have that $\cL$ is  path connected and hence, by  assumption
has finite finitistic dimension.  We  also have that $\cM$ has $k-1$	
path connected components and, by the
induction hypothesis, has finite finitistic dimension.  The result now follows from
Theorem \ref{thm-fgr}. \qed
\vskip .2in

We end with a final application.  We say a cycle is \emph{simple} if it has no repeated
vertices.   We say a quiver $\cQ$ is of \emph{simple cycle type} if every nontrivial
path connected component is a simple cycle.

\begin{prop}\label{prop-simple} If $\cQ$ is of simple cycle type, then
the finitistic dimension of $K\cQ/I$ is finite, for all admissible ideals $I$.
\end{prop}

\begin{proof}
We note that if $C$ is a simple cycle then $KC/L$ has finite
finitistic dimension for any admissible ideal $L$ since
$KC/L$ is a monomial algebra \cite{GKK}. Thus, if $\cQ$ is
of simple cycle type, $I$ is an  admissible ideal  in
$K\cQ$, and $C$ is a nontrivial path connected component,
then $KC/(KC\cap I)$ has finite finitistic dimension. The remainder of the 
proof is mirrors the  proof of  Theorem \ref{thm-appl}.

\end{proof}

%****************
\bibliographystyle{plain}

\end{document}